\documentclass[12pt,a4paper]{amsart}

\usepackage{amssymb,amscd,latexsym,amsxtra,dsfont,enumerate}

\usepackage[markup=underlined]{changes}
\definechangesauthor[color=red]{Curt}
\definechangesauthor[color=blue]{Emanuel}
\usepackage{amsfonts,bbold}
\usepackage[mathscr]{eucal}
\usepackage{calrsfs}
\usepackage{fge}
\usepackage{mathtools}
\usepackage{scalerel}
\usepackage[autosize]{dot2texi}

\definecolor{britishracinggreen}{rgb}{0.0, 0.26, 0.15}
\definecolor{amaranth}{rgb}{0.9, 0.17, 0.31}

\newtheorem{thm}{Theorem}

\newtheorem{lem}[thm]{Lemma}
\newtheorem{prop}[thm]{Proposition}

\newtheorem{remark}{Remark}

\theoremstyle{definition}
\newtheorem{defn}[thm]{Definition}

\newtheorem{prob}{Problem}

\newcommand{\sss}[1]{{\scriptscriptstyle #1}}

\newcommand\restr[2]{{
		\left.\kern-\nulldelimiterspace 
		#1 
		\vphantom{\big|} 
		\right|_{#2} 
}}

\newcommand{\mynorm}[1]{ \left\| #1 \right\| }

\begin{document}

	\title{Extending surjective maps preserving the norm of symmetric Kubo-Ando means}
	
	\author{Emmanuel Chetcuti}
	\address{
		Emmanuel Chetcuti,
		Department of Mathematics\\
		Faculty of Science\\
		University of Malta\\
		Msida MSD 2080  Malta} \email {emanuel.chetcuti@um.edu.mt}
	
	\author{Curt Healey}
	\address{Curt Healey\\
		Department of Mathematics\\
		Faculty of Science\\
		University of Malta\\
		Msida MSD 2080  Malta}
	\email{curt.c.healey.13@um.edu.mt}

	\date{\today}
	\begin{abstract}
		In  \cite{dong-molnar-wong}, the authors addressed the question of whether surjective maps preserving the norm of a symmetric Kubo-Ando mean can be extended to Jordan $\ast$-isomorphisms. The question was affirmatively answered for surjective maps  between $C^{*}$-algebras for certain specific classes of symmetric Kubo-Ando means. Here, we give a comprehensive answer to this question for surjective maps between $AW^{*}$-algebras  preserving the norm of any symmetric Kubo-Ando mean.
	\end{abstract}
	\subjclass[2000]{Primary 47A64, 47B49, 46L40}
	\keywords{Kubo-Ando connection, $C^{*}$-algebra,  $AW^{*}$-algebra, Positive definite cone, Order, Preservers}
	\maketitle

	\section{Introduction}
	
	Recently, in \cite{dong-molnar-wong}, considerable attention was given to the 
problem of characterizing those maps between the positive definite cones of unital $C^{*}$-algebras that preserve the norm of a given Kubo-Ando mean.  We recall that a binary operation $\sigma$ on the positive definite cone $\mathcal B(H)^{++}$ of the algebra $\mathcal B(H)$ of bounded operators on the Hilbert space $H$,  is called a \emph{Kubo-Ando connection} if it satisfies the following properties:
	\begin{enumerate}[{\rm(i)}]
		\item If $A \leq C$ and $B \leq D$, then $ A \sigma B \leq C \sigma D $.
		\item $ C (A \sigma B) C \leq (CAC) \sigma (CBC) $.
		\item  If $ A_n  \downarrow A$ and $ B_n \downarrow B $, then $ A_n \sigma B_n \downarrow A \sigma B$\footnote{For a sequence $(X_n)$ of self-adjoint operators in $\mathcal B(H)$, we write $X_n\downarrow X$ when $(X_n)$ is monotonic decreasing and SOT-convergent to $X$.  The symbol $X_n\uparrow X$ is defined dually.}.
	\end{enumerate}
	A \emph{Kubo-Ando mean} is a Kubo-Ando connection with the normalization condition $I\sigma I=I$.  The most fundamental connections are:
	\begin{itemize}
		\item  the \emph{sum} $(A,B)\mapsto A+B$,
		\item the \emph{parallel sum} $(A,B)\mapsto A:B=\left(A^{-1}+B^{-1}\right)^{-1}$,
		\item the \emph{geometric mean}
		\[(A,B)\mapsto A\sharp B=A^{\sss{\frac{1}{2}}}\,\left(A^{\sss{-\frac{1}{2}}}BA^{\sss{-\frac{1}{2}}}\right)^{\sss{\frac{1}{2}}}A^{\sss{\frac{1}{2}}}.\]
	\end{itemize}
		
	The domain of definition can easily be extended from $\mathcal B(H)^{++}$ to the positive semi-definite cone $\mathcal B(H)^{+}$. For details, refer to the introduction section in \cite{curt-emanuel-first-paper}.

	 In \cite[Theorem  3.2]{kubo-ando-main}, it is shown that there is an affine order isomorphism from the class of Kubo-Ando connections onto the class of operator monotone functions via the map $f(xI) = I \sigma (xI)$ for $x>0$. Moreover, it is also shown that $f(A)=I\sigma A$ for every $A\in\mathcal B(H)^{+}$, which implies that
	\[A\sigma B=A^{\sss{\frac{1}{2}}}\,f\left(A^{\sss{-\frac{1}{2}}}BA^{\sss{-\frac{1}{2}}}\right)A^{\sss{\frac{1}{2}}}, \quad  \forall \, A \in  \mathcal B(H)^{++}, B \in \mathcal B(H)^{+}.  \]
	The function $f$ is called the \emph{representing function} of $\sigma$.  We further recall that if $\sigma$ is a Kubo-Ando connection with representing function $f$, then the representing function of the `reversed' Kubo-Ando connection $(A,B)\mapsto B\sigma A$ is the transpose $f^\circ$,  defined by $f^\circ(x):=xf(x^{-1})$. The Kubo-Ando connection is said to be symmetric if it coincides with its reverse; that is, a Kubo-Ando connection is symmetric if and only if the representing function $f$ satisfies $f=f^\circ$ as shown in \cite[Corollary 4.2]{kubo-ando-main}. The Kubo-Ando means are precisely the Kubo-Ando connections whose representing function satisfy the normalizing condition $f(1)=1$.

	We recall that operator monotone functions correspond to positive finite Borel measures on $[0,\infty]$\footnote{We recall that every finite Borel measure on $[0,\infty]$ is regular, i.e. a Radon measure.} by L\"owner's Theorem (see \cite{MR0486556}): To every  operator monotone function $f$ corresponds a unique positive and finite Borel measure $m$ on $[0,\infty]$ such that
	\begin{equation}\label{e4}
		f(x)=\,\int_{[0,\infty]}\frac{x(1+t)}{x+t}\,{\rm d}m(t)\,=\,m(\{0\})\,+\,x\,m(\{\infty\})\,+\,\int_{(0,\infty)}\frac{1+t}{t}(t:x)\,{\rm d}m(t)\quad(x>0).
	\end{equation}
	where $t:x = 2^{-1} (t \,! \, x)$. It is easy to see that $f(0+)=m(\{0\})$, $f^\circ(0+)=m(\{\infty\})$.
	Finally, by \cite[Theorem 3.4]{kubo-ando-main}, there is an affine isomorphism from the class of positive finite Borel measures on $[0,\infty]$ onto the class of Kubo-Ando connections. This is given by the formula
	\begin{equation}\label{kubo-ando-integral-representation}
		A \sigma B = \alpha A +  \beta B + \int_{(0, \infty)} \frac{1+t}{t} (tA:B) \, {\rm d} m(t) \quad A, B \in  \mathcal B(H)^{+}
	\end{equation}
	where $\alpha = m(\lbrace0\rbrace)$ and $\beta =m(\lbrace \infty \rbrace)$.
	In the case of a symmetric Kubo-Ando connection, then $ \alpha = \beta $. For further details on the provenance of the integral representation (\ref{kubo-ando-integral-representation}), the reader is referred to \cite[Theorem 3.2]{kubo-ando-main}.  
	
	After the exposition on general Kubo-Ando means, we can now define the property under study for symmetric Kubo-Ando means.
	
	\begin{defn}\label{phi-definition}
		Let $\sigma$ be a Kubo-Ando mean and let  $\mathcal{A}$ and $\mathcal{B}$ be unital $C^{*}$-subalgebras of $\mathcal B(H)$.  A surjective map $\phi$ between the positive definite cones of $\mathcal{A}$ and $\mathcal{B}$ is said to preserve the norm of  $\sigma$ if
		\[  \mynorm{A \sigma B} = \mynorm{\phi(A) \sigma \phi(B) }, \quad \forall \,  A, B \in \mathcal{A}^{++}.
		\]
	\end{defn}
	
	 A natural question to ask is whether a surjective map $\phi : \mathcal{A}^{++} \rightarrow \mathcal{B}^{++}$ preserving the norm of a symmetric Kubo-Ando mean is an order isomorphism.  By an \emph{order isomorphism}, we mean a  map $\phi$ such that $A \leq B \iff  \phi(A) \leq \phi(B)$ where $A, B \in \mathcal{A}^{++}$. By \cite[Theorem 6]{curt-emanuel-first-paper}, we can easily show this: let $A, B \in \mathcal{A}^{++}$, then
	\begin{equation}\label{phi-order-isomorphism}
		\begin{split}
			A \leq B  &\iff \mynorm{A \sigma X} \leq \mynorm{B \sigma X},  \quad \forall \; X \in \mathcal{A}^{++}  \\
			&\iff \mynorm{\phi(A) \sigma \phi(X)} \leq \mynorm{\phi(B) \sigma \phi(X)},  \quad \forall \; X \in \mathcal{A}^{++} \\
			&\iff \phi(A) \leq \phi(B).
		\end{split}
	\end{equation}	
	Moreover, $\phi$ is norm preserving. This can be seen by recalling \cite[Theorem 3.3]{kubo-ando-main} and noting that	
	\begin{equation}\label{phi-norm-preserving}
		\mynorm{A} = \mynorm{A \sigma A} = \mynorm{\phi(A) \sigma \phi(A) } = \mynorm{\phi(A)}, \quad \forall \,  A \in \mathcal{A}^{++}.
	\end{equation}
	
	The question we now tackle is whether a surjective map preserving the norm of a symmetric Kubo-Ando mean extends to a Jordan $\ast$-isomorphism. Let us recall that a \emph{Jordan $\ast$-isomorphism}  is a bijective linear map $J: \mathcal{A} \rightarrow \mathcal{B}$ such that $J(A^{*}) = J(A)^{*}$ and $J(AB + BA) = J(A)J(B) + J(B)J(A)$ for $A,B \in \mathcal{A}$. The problem under study has been stated explicitly in the open problem section of \cite{molnar-od}, and we reformulate it here to its most general form.
	\begin{prob}
		Do surjective maps between positive cones of unital $C^{*}$-algebras that preserve the norm of a symmetric Kubo-Ando mean extend to Jordan $\ast$-isomorphisms?
	\end{prob}
	
 The above problem has been solved for the arithmetic mean \cite[Theorem 2.4]{dong-molnar-wong} and the geometric mean \cite[Theorem 1]{chabbabi-main}, but for the harmonic mean, it has been solved in the case of $AW^{*}$-algebras \cite[Theorem 2.16]{dong-molnar-wong}. Our aim is to provide a complete answer for  to the above problem for general symmetric means in the setting of $AW^{*}$-algebras.

\section{Preliminary Considerations}
	Let us first recall \cite[Lemma 2.3]{dong-molnar-wong} and provide the proof here for completeness' sake.
	\begin{lem}\label{extending-phi-to-semi-definite-cone}
		Let $\mathcal{A}, \mathcal{B}$ be unital $C^{*}$-algebras and  $\phi: \mathcal{A}^{++} \rightarrow \mathcal{B}^{++}$ a surjective norm preserving order isomorphism, then $\phi(tI) = tI$ for $t > 0$.
	\end{lem}
	
	\begin{proof}
		Given that $ \mynorm{\phi(tI)} = \mynorm{tI} = t$, it follows that $\phi(tI) \leq tI$. Furthermore, $\exists \; A \in \mathcal{A}^{++}$ such that $\phi(A) = tI$. Since $  \mynorm{A} = \mynorm{\phi(A)} = t$, it implies that  $A \leq tI$. Hence, we have $\phi(A) \leq \phi(tI)$ which proves the lemma.
	\end{proof}

 Moreover, by \cite[Lemma 2.8]{dong-molnar-wong},  $\psi_{\epsilon}(A) = \phi(A+ \epsilon I) - \epsilon I$ is a surjective norm preserving order isomorphism between the positive semi-definite cones of $\mathcal{A}$ and $\mathcal{B}$ for some $\epsilon > 0$ such that $\psi_{\epsilon}(tI) = tI$ for all $t > 0$. Let us provide the proof here for completeness' sake. 
 
 \begin{lem}\label{psi-lemma}
 			Let $\mathcal{A}, \mathcal{B}$ be unital $C^{*}$-algebras and  $\phi: \mathcal{A}^{++} \rightarrow \mathcal{B}^{++}$ a surjective norm preserving order isomorphism such that $\phi(tI) = tI$ for all $t > 0$. Define $\psi_{\epsilon}: \mathcal{A}^{+} \rightarrow \mathcal{B}^{+}$ by $\psi_{\epsilon}(A) = \phi(A + \epsilon I) - \epsilon I$ where $\epsilon > 0$, then $\psi_{\epsilon}$ is a surjective norm preserving order isomorphism such that $\psi_{\epsilon}(tI) = tI$ for all $t > 0$.
 \end{lem}
 
 \begin{proof}
 	Let $A, B \in \mathcal{A}^{+}$, and $\epsilon > 0$, then
		\begin{equation*}
		\begin{split}
			A \leq B  &\iff \phi(A + \epsilon I) \leq \phi(B + \epsilon I) \\
			&\iff \phi(A + \epsilon I) - \epsilon I \leq \phi(B + \epsilon I) - \epsilon I \\
			&\iff \psi_{\epsilon}(A) \leq \psi_{\epsilon}(B)
		\end{split}
	\end{equation*}	
	which implies that $\psi_\epsilon$ is an order isomorphism. We next show that $\psi_{\epsilon}$ is surjective. Let $B \in \mathcal{B}^{+}$, then there is some $A \in \mathcal{A}^{++}$ such that $\phi(A) = B + \epsilon I$. Since $\phi(A) \geq \epsilon I$, then $A \geq \epsilon I$, which allows us to conclude that  $\psi_{\epsilon} (A - \epsilon I) = \phi(A) - \epsilon I = B $. Therefore, $\psi_{\epsilon}$ is surjective. 
	Since $\mynorm{\psi_{\epsilon}(A)} = \mynorm{\phi(A + \epsilon I)  - \epsilon I} = \mynorm{\phi(A + \epsilon I)} - \epsilon I = \mynorm{A + \epsilon I} - \epsilon I = \mynorm{A}$, then $\psi_{\epsilon}$ is norm preserving. Finally, $\psi(tI) = \phi( (t + \epsilon) I) - tI = (t + \epsilon) I - \epsilon I = tI$.
	
 \end{proof}

	\par 
		Before recalling \cite[Lemma 3.1]{mori}, let us provide some definitions.  For a unital $C^{*}$-algebra $\mathcal{A}$ we define  $E_t(\mathcal{A}) := \lbrace 0 \leq A \leq tI \rbrace$, where $t$ is a positive real number.  We recall $E_1(\mathcal{A})$ is called the effect algebra associated with $\mathcal A$. Let 
		\[ \Delta_t(\mathcal{A}) := \lbrace (a, b) \in E_t(\mathcal{A}) \times E_t(\mathcal{A}):  a \wedge b = 0 \rbrace\]
		where $a \wedge b$ denotes the infimum of $a, b$ in $\mathcal{A}^{+}$. Finally, denote the set of projections of $\mathcal{A}$ by $P(\mathcal{A})$.  We recall that for any non-empty family of projections in $\mathcal{A}$, if the infimum in $P(\mathcal{A})$ exists, then this will be also the infimum of the family, taken in 
$\mathcal{A}^{+}$.
	\begin{lem}\label{mori-lemma}
		Let $\mathcal{A}$ be a unital $C^{*}$-algebra. Endow $\Delta_1(\mathcal{A})$ with an order relation such that $(a_1, a_2) \leq (b_1, b_2) \iff a_1 \leq b_1$ and $b_2 \leq a_2$, where $(a_1, a_2), (b_1, b_2) \in \Delta_1(\mathcal{A})$. Then the following conditions are equivalent:
		\begin{enumerate}[{\rm(1)}]
			\item $P$ is a projection. 
			\item The pair $(P, B)$ is maximal in $\Delta_1(\mathcal{A})$ for some $B \in E_1(\mathcal{A})$. 
		\end{enumerate}
	\end{lem}
	
	\begin{proof}
		(1) $\Rightarrow$ (2). Suppose that $P$ is a projection. Consider the pair $(P, I-P)$. Let $A \in \mathcal{A}^{+}$ such that $A \leq P$ and $A \leq I-P$. Then $(I-P) A (I-P) \leq (I-P) P (I-P) = 0$. Similarly, $PAP = 0$  which implies $A = (I-P)AP + PA(I-P)$, but since $(I-P)A=A$, and $PA =A$, then $A = 2A$, so $A=0$.
		\par 
		Suppose $(P, I-P) \leq (A, B)$ where $(A, B) \in \Delta_1(\mathcal{A})$. Since $P \leq A \leq I$, then $AP = P$. Thus, $(I-P)AP = 0 $ and $PA(I-P) = 0$. Since $A = P + (I-P)A(I-P)$, then $(I-P)A(I-P) \leq A$ and $(I-P)A(I-P) \leq (I-P) \leq B $ which implies that $(I-P)A(I-P) = 0$. Therefore, $A=P$, and, similarly, it can be concluded that $B = I-P$.
		\par
		(2) $\Rightarrow$ (1). If $(P, B)$ is maximal, then $(P, B) = (P^{1/2}, B^{1/2})$ which implies that $\text{sp}(P) \subseteq \lbrace 0, 1 \rbrace$, so it must be a projection.
	\end{proof}
	
	\begin{remark}\label{phi-maps-projections-to-projections}
		 Let $\mathcal{A}, \mathcal{B}$ be unital $C^{*}$-algebras and let $\psi: \mathcal{A}^{+} \rightarrow \mathcal{B}^{+}$ be a surjective order isomorphism such that $\psi(tI) = tI$, then $(A, B)$ is maximal in $\Delta_t(\mathcal{A})$ iff $(\psi(A), \psi(B))$ is maximal in $\Delta_t(\mathcal{B})$. In particular, $P \in P(\mathcal{A})$ iff $\psi(P) \in P(\mathcal{B})$ by Lemma \ref{mori-lemma}.
	\end{remark}

	\begin{lem}\label{psi-positive-homogenous}
		Let $\mathcal{A}, \mathcal{B}$ be unital $C^{*}$-algebras. If $\psi: \mathcal{A}^{+} \rightarrow \mathcal{B}^{+}$ is a surjective order isomorphism such that $\psi(tI) = tI$ for all $t > 0$, then for any projection $P$ in $\mathcal{A}$,  $\psi(t P) =t \psi(P)$ for all $t > 0.$ 
	\end{lem}
	
	\begin{proof}
		Let $\Delta_t(\mathcal{A})$
		be endowed with the same order relation as specified for $\Delta_1(\mathcal{A})$.
		We claim that  $(P_1, P_2)$ is maximal in $\Delta_1(\mathcal{A})$ if and only if $(tP_1, tP_2)$ is maximal in $\Delta_t(\mathcal{A})$. Let $(P_1, P_2)$ be maximal in $\Delta_1(\mathcal{A})$. That $P_1 \wedge P_2 = 0$ $\iff$ $tP_1 \wedge tP_2 = 0$ is clear. Suppose $(A, B) \in \Delta_t(\mathcal{A})$ such that $(tP_1, tP_2) \leq (A, B)$, then $(P_1, P_2) \leq (t^{-1}A, t^{-1}B  )$ and $(t^{-1}A, t^{-1}B  ) \in \Delta_1(\mathcal{A})$ which by maximality of $(P_1, P_2)$ implies that $(P_1, P_2) = (t^{-1}A, t^{-1}B  )$, so $(tP_1, tP_2) = (A, B)$. Similar arguments are used to prove the reverse implication.
		\par 
		 Let $t > 0$, since $(P, (I-P))$ is maximal in $\Delta_1(\mathcal{A})$, then $(tP, t(I-P))$ is maximal in $\Delta_t(\mathcal{A})$ which implies that $(\psi(tP), \psi(t(I-P)))$ is maximal in $\Delta_{t}(\mathcal{B})$ by Remark $\ref{phi-maps-projections-to-projections}$, then $(t^{-1}\psi(tP), t^{-1}\psi(t(I-P)))$ is maximal in $\Delta_{1}(\mathcal{B})$, so $(t^{-1}\psi(tP), t^{-1}\psi(t(I-P))) \in P(\mathcal{B}) \times P(\mathcal{B})$.
		 \par
		  If $t > 1$, then $(\psi(P), \psi(I-P))  \leq (\psi(tP), \psi(t(I-P)))$. Since $\text{rng}(\psi(P)) \subseteq \text{rng}(\psi(tP)) = \text{rng}(t^{-1}\psi(tP))$ and $\text{rng}(\psi(I-P)) \subseteq \text{rng}(t^{-1}\psi(t(I-P)))$, then $(\psi(P), \psi(I-P))  \leq (t^{-1}\psi(tP), t^{-1}\psi(t(I-P)))$, so by maximality $\psi(P) = t^{-1} \psi(tP)$. 
		  \par
		  If $t < 1$, then 
		 $(\psi(tP), \psi(t(I-P)))  \leq (\psi(P), \psi(I-P))$. Since $\text{rng}(t^{-1}\psi(tP)) =\text{rng}(\psi(tP)) \subseteq \text{rng}(\psi(P))$ and $\text{rng}(t^{-1}\psi(t(I-P))) \subseteq \text{rng}(\psi(I-P))$, then $ (t^{-1}\psi(tP), t^{-1}\psi(t(I-P)) \leq (\psi(P), \psi(I-P))$, so by maximality $\psi(P) = t^{-1} \psi(tP)$.

	\end{proof}

	Let $\epsilon_1 < \epsilon_2$ and consider maps $\psi_{\epsilon_1}, \psi_{\epsilon_2}$ of the type stated in Lemma $\ref{psi-lemma}$. Then
	\[  \phi(tP + \epsilon_1I) \leq \phi(tP + \epsilon_2I), \quad \forall \,   P \in P(\mathcal{A}), \, t > 0. \]
	By Lemma \ref{psi-positive-homogenous},
	\[ t\psi_{\epsilon_1}(P) + \epsilon_1I \leq t \psi_{\epsilon_2}(P) + \epsilon_2I, \quad \forall \, P \in P(\mathcal{A}), \, t > 0. \]
	As $t \rightarrow \infty$,  
	\begin{equation}\label{psi-family-projection-inequality}
		\psi_{\epsilon_1}(P)  \leq \psi_{\epsilon_2}(P), \quad \forall \, P \in P(\mathcal{A}).
	\end{equation}
	
	Moreover, we claim that $\psi_{\epsilon_1} (P) = \psi_{\epsilon_2}(P)$. Let $Q = \psi_{\epsilon_2}(P) - \psi_{\epsilon_1}(P)$, then $Q$ is a projection which is orthogonal to $\psi_{\epsilon_1}(P)$. Furthermore, $\psi_{\epsilon_2}^{-1}(Q) = R  \leq P$, which implies that $\psi_{\epsilon_1}(R) \leq \psi_{\epsilon_1}(P)$. By Remark \ref{phi-maps-projections-to-projections}, $R$ is a projection, so  (\ref{psi-family-projection-inequality}) implies $\psi_{\epsilon_1}(R) \leq \psi_{\epsilon_2}(R) = Q$ which implies that $ \psi_{\epsilon_1}(R) = 0$. Since $\psi_{\epsilon}(0) = 0$ for any $ \epsilon > 0$, then $Q = 0$, and the claim is proved.

	Subsequently,
	\[ \mynorm{A \sigma \big( P + \epsilon I \big)} = \mynorm{\phi(A) \sigma \big( \psi_{\epsilon}(P) + \epsilon I \big)},  \quad \forall \,  \epsilon > 0. \]
	Since $\psi_{\epsilon}(P) = Q$ for all $\epsilon > 0$, then
 	as $\epsilon \rightarrow 0$
	\begin{equation*}
		\mynorm{A \sigma P} = \mynorm{\phi(A) \sigma Q}
	\end{equation*}
		by \cite[Remark 1 (i)]{curt-emanuel-first-paper}. This implies that the above can be written as
	\begin{equation}\label{key-equation}
		\mynorm{A \sigma P} = \mynorm{\phi(A) \sigma \psi_\epsilon(P)}, \quad \forall \,  \epsilon > 0.
	\end{equation}
	Furthermore, if $A = Q + \delta I$ for some $Q \in P(\mathcal{A})$ and $\delta > 0$, then 
	\begin{equation}\label{key-equation-v2}
		\mynorm{ \big( Q + \delta I \big) \sigma P} = \mynorm{ \big( \psi_{\epsilon}(Q) + \delta I \big) \sigma  \psi_{\epsilon}(P)}, \quad \forall \,  \epsilon > 0.
	\end{equation}
	
	\section{Results}

	\begin{thm}\label{MAIN}
		Let $\mathcal{A}, \mathcal{B}$ be $AW^{*}$-algebras. A surjective map $\phi: \mathcal{A}^{++} \rightarrow \mathcal{B}^{++}$ preserves the norm of a symmetric Kubo-Ando mean $\sigma$ if and only if there is a Jordan $\ast$-isomorphism $J : \mathcal{A} \rightarrow \mathcal{B}$ which extends $\phi$, i.e. $\phi(A) = J(A)$ holds for all $A \in \mathcal{A}^{++}$.
	\end{thm}
	
Sufficiency is trivial.  We prove that if $\phi: \mathcal{A}^{++} \rightarrow \mathcal{B}^{++}$ is a surjective map 
satisfying $\Vert A\sigma B\Vert=\Vert \phi(A)\,\sigma\,\phi(B)\Vert$ for every $A,B\in\mathcal B^{++}(H)$, then there is a Jordan $\ast$-isomorphism $J : \mathcal{A} \rightarrow \mathcal{B}$ satisfying $\phi(A) = J(A)$ for all $A \in \mathcal{A}^{++}$.  The proof shall be split into two cases depending on the behaviour of the representation function $f$ at $0$. \\

\emph{Case 1: $f(0+)=0$}.  For this case we make use of the following characterization of positive homogeneous order isomorphisms\footnote{i.e. $\phi(tA) = t \phi(A)$ for $A \in \mathcal{A}^{++}$ and $t > 0$. }.

\begin{thm}{\cite[Theorem 13]{molnar-quantum}}\label{jordan-extension-for-positive-homogenous}
		Let $\mathcal{A}$, $\mathcal{B}$ be unital $C^{*}$-algebras. The map $\phi: \mathcal{A}^{++} \rightarrow \mathcal{B}^{++}$ is a surjective positive homogenous order isomorphism if and only if it is of the form 
		\begin{equation}\label{jordan-isomorphim-representation-quantum-renyi}
			\phi(A) = C J(A) C, \quad \forall \, A \in \mathcal{A}^{++}
		\end{equation}
		where $C \in \mathcal{B}^{++}$ and $J: \mathcal{A} \rightarrow \mathcal{B}$ is a Jordan $*$-isomorphism.
	\end{thm}
	It is clear in this theorem that if $\phi(I)=I$ then $C=I$.  We further recall the following  characterisation of Kubo-Ando connections with representing function $f$ satisfying $f(0+) = 0$.
	
	\begin{prop}{\cite[Lemma 2]{molnar-od}}\label{molnar-function-at-zero}
		Let $f:(0, \infty)  \rightarrow (0, \infty) $ be a non-trivial (i.e. not affine) operator monotone  function satisfying $f(0+)=0$ and let $\sigma$ denote the Kubo-Ando connection associated to $f$.  For  $A\in \mathcal B(H)^{++}$ and non-zero projection $P\in \mathcal B(H)$ 
		\[\mynorm{A\sigma P}=f^\circ\left(\frac{1}{\max\{\lambda\ge0:\lambda P\le PA^{-1}P\}}\right).\]
	\end{prop}
	
	Since the geometric mean is a symmetric mean with representation function $f$ such that $f(t) = t^{1/2}$,  then
	\begin{equation}\label{geometric-mean-max-pap}
		\mynorm{A \# P}^{2} = \frac{1}{\max\{\lambda\ge0:\lambda P\le PA^{-1}P\}}.
	\end{equation}
	Therefore, 
	\begin{equation}\label{kubo-ando-geometric-relationship}
		\mynorm{A \sigma P} = f^{\circ} ( \mynorm{A \# P}^{2}).
	\end{equation}

	 We shall also require the following result.

	\begin{prop}{\cite[Lemma 11]{chabbabi-main}}\label{geometric-mean-od-awstar-algebra}
		Let $\mathcal{A}$ be an $AW^{*}$-algebra and $\sigma$ a symmetric Kubo-Ando connection with corresponding representation function $f $such that $f(0+)=0$. Suppose $A, B \in \mathcal{A}^{++}$, then
		\[ A \leq B \iff \mynorm{A \sigma P} \leq \mynorm{B \sigma P}, \quad \forall \, P \in P(AW^{*}(I, A^{-1}-B^{-1})). \ \]
	\end{prop}
	\begin{proof}
		Let $T = A^{-1}- B^{-1}$  and consider $AW^{*}(I, T)$. It must be first recalled that  any commutative $AW^{*}$-algebra is algebra $\ast$-isomorphic to some $C(X)$ where $X$ is compact, Hausdorff, and extremally disconnected$\footnote{An extremally disconnected set is a set such that the closure of every open set is open.}$ \cite[Theorem 1 Section 7]{berberian}. \par
	    Therefore, let $f_T$ be the corresponding function of $T$ in $C(X)$. If $T$ is not positive then the spectrum $\sigma(T)$ contains some negative number, so there is some $\epsilon$ such that $\sigma(T) \cap ]-\infty, -\epsilon[  \neq \emptyset$. Consider the projection $P_\epsilon$ associated with $\overline{f_{T}^{-1}(-\infty, -\epsilon)}$. Since $f_T ( \overline{f_{T}^{-1}(-\infty, -\epsilon)} ) \subseteq \overline{f_T ({f_{T}^{-1}(-\infty, -\epsilon)} )} \subseteq (-\infty, -\epsilon]$, then  $P_\epsilon T P_\epsilon \leq -\epsilon P_{\epsilon} $, so $P_\epsilon A^{-1} P_\epsilon + \epsilon P_{\epsilon} \leq P_\epsilon B^{-1} P_\epsilon $. Therefore, for $\lambda > 0$
	    \[  \lambda P_\epsilon \leq P_\epsilon A^{-1} P_\epsilon \implies (\lambda + \epsilon)P_\epsilon \leq P_\epsilon B^{-1} P_\epsilon.  \]
	    
	     Furthermore, since $\mynorm{A \# P} \neq 0$ for any  $P \in P(\mathcal{A})$, then by (\ref{geometric-mean-max-pap}), (\ref{kubo-ando-geometric-relationship}), and the injectivity of $f$, it can be concluded that $\mynorm{B \sigma P_{\epsilon}} < \mynorm{A \sigma P_{\epsilon}}$ which is a contradiction.
	\end{proof}
	
Now we complete the proof for the case when $f(0+)=0$ by showing that the map $\phi$ of Theorem \ref{MAIN} is positive homogeneous.
		
\begin{proof}
		Let $A \in \mathcal{A}^{++}$, $P \in P(\mathcal{A})$, and $ \epsilon > 0$, then by  (\ref{key-equation}) 
		\[  \mynorm{A \sigma P } = \mynorm{ \phi(A) \sigma  \psi_{\epsilon}(P)}. \]
	 By (\ref{kubo-ando-geometric-relationship}) and the fact that $f$ is injective,
		\begin{equation}\label{positive-homogeonous-p1}
			\mynorm{A \# P} = \mynorm{ \phi(A) \# \psi_{\epsilon}(P)}.
		\end{equation}
	 Let $t > 0$, 
	 \begin{equation}\label{positive-homogeonous-p2}
	 	\mynorm{tA \# P} = t^{1/2}\mynorm{A \# P} = t^{1/2}\mynorm{\phi(A)  \# \psi_{\epsilon}(P)} = \mynorm{t\phi(A) \# \psi_{\epsilon}(P) }.
	 \end{equation}
	By (\ref{positive-homogeonous-p1}) and (\ref{positive-homogeonous-p2}), 
	\[ \mynorm{ \phi(tA) \# \psi_{\epsilon}(P)} = \mynorm{t\phi(A) \# \psi_{\epsilon}(P) }. \]
	By Proposition \ref{geometric-mean-od-awstar-algebra}, we can then conclude that $t\phi(A) = \phi(tA)$ which proves that $\phi$ is positive homogenous. 
	\end{proof}

	\emph{Case 2: $f(0+)>0$}. We will show that $\psi_{\epsilon}$ is orthogonality preserving (in both directions) for any $\epsilon > 0$\footnote{i.e. if $A, B \in \mathcal{A}^{+}$, then $AB = 0 \iff \psi_{\epsilon}(A) \psi_{\epsilon}(B) =0$}. Furthermore, \cite[Lemma 2.3]{dong-molnar-wong} will allow us to characterize surjective maps which are orthogonality and norm preserving order isomorphisms:
	
	\begin{lem}{\cite[Lemma 2.3]{dong-molnar-wong}}
	Let $\mathcal{A}, \mathcal{B}$ be $C^{*}$-algebras such that at least one is unital. Let $\psi_{\epsilon} : \mathcal{A}^{+} \rightarrow \mathcal{B}^{+}$ be a surjective order isomorphism such that $\psi_{\epsilon}$ is norm and orthogonality preserving, then $\psi_{\epsilon}$ extends to a Jordan $\ast$-isomorphism $J: \mathcal{A} \rightarrow \mathcal{B}$.
	\end{lem}
	
	Before proceeding with the proof, let us recall \cite[Lemma 1]{curt-emanuel-first-paper}.
	
	\begin{lem}\label{kubo-ando-decomposition-at-zero}
		Let $f:(0, \infty)  \rightarrow (0, \infty) $ be an operator monotone  function and let $m$ denote the positive and finite Borel measure associated to $f$ via (\ref{e4}). For every Borel subset $\Delta$ of $[0,\infty]$ satisfying $m(\Delta)>0$, the function $f_\Delta$ defined on $(0,\infty)$  by
		\[f_\Delta:x\mapsto\int_{\Delta}\frac{x(1+t)}{x+t}\,{\rm d}m(t)\]
		is operator monotone.  In particular, if $m((0,\infty))\neq 0$, the function $h$ defined by
		\[h(x):=\int_{(0,\infty)}\frac{x(1+t)}{x+t}\,{\rm d}m(t)=f(x)-f(0+)-f^\circ(0+) x\quad(x>0)\]
		is operator monotone.  If $f$ is symmetric, then so is $h$.
	\end{lem}
	
	Let us now prove Theorem \ref{MAIN} for the case when $f(0+)>0$.
	
	\begin{proof}
	Let $P, Q \in P(\mathcal{A})$ be two orthogonal projections, and consider $(Q + \delta I) \sigma P$ where $\delta > 0$. By (\ref{key-equation-v2}), the following equation is obtained
	\begin{equation}\label{beginning-equation-of-final-theorem}
		\mynorm{ \big(Q + \delta I\big) \sigma P} =  \mynorm{ \big(\psi_{\epsilon}(Q) + \delta I \big) \sigma \psi_{\epsilon}(P)}, \quad \forall \, \epsilon > 0.
	\end{equation}

	If $m((0,\infty))\neq 0$, by Lemma \ref{kubo-ando-decomposition-at-zero} we can denote $\sigma_h$ as the symmetric Kubo-Ando connection corresponding to $h(x) = f(x) - \alpha - \alpha x$ where $\alpha = f(0+)$. By (\ref{kubo-ando-integral-representation}), we can decompose $(Q + \delta I) \sigma P$ in the following way:
	\[  (Q + \delta I) \sigma P = \alpha(Q + \delta I + P) + (Q+\delta I) \sigma_h P.   \] 
	
	 Since $P(Q + \delta I)^{-1}P  =  P\big( (1+\delta)^{-1}Q + \delta^{-1}(I-Q)\big)P =   \delta^{-1}P $, then by Proposition \ref{molnar-function-at-zero} it can be concluded that
	\[   \mynorm{ \big(Q + \delta I\big) \sigma_h P} = h(\delta) \quad  \text{since} \quad \max\{\lambda\ge0:\lambda P\le P(Q+ \delta I)^{-1}P\} = \delta^{-1}. \]
	
	Therefore, 
	\begin{equation}\label{left-hand-side-for-final-theorem}
		\begin{split}
			\mynorm{(Q + \delta I) \sigma P} & = \mynorm{ \alpha(Q + \delta I + P) +  (Q + \delta I) \sigma_h P } \\ 
			& \leq \mynorm{ \alpha(Q + \delta I + P) +  \mynorm{(Q + \delta I) \sigma_h P} I } \\
			& \leq \mynorm{ \alpha(Q + \delta I + P) +  h(\delta) I }.
		\end{split}
	\end{equation}

	Furthermore, 
	\begin{equation}\label{right-hand-side-for-final-theorem}
		\begin{split}
		\mynorm{ \big(\psi_{\epsilon}(Q) + \delta I \big) \sigma \psi_{\epsilon}(P)} =& \mynorm{ \alpha \big(\psi_{\epsilon}(Q) + \delta I + \psi_{\epsilon}(P)\big) + (\psi_{\epsilon}(Q)+\delta I) \sigma_h \psi_{\epsilon}(P)} \\
		\geq& \mynorm{\alpha(\psi_{\epsilon}(Q) + \psi_{\epsilon}(P))} 
		\end{split}
	\end{equation}
	where the above follows because $  (\psi_{\epsilon}(Q)+\delta I) \sigma_h \psi_{\epsilon}(P)  $ is a positive operator. Therefore, by using the inequalities (\ref{left-hand-side-for-final-theorem}) and (\ref{right-hand-side-for-final-theorem}) in equation (\ref{beginning-equation-of-final-theorem}), 
	\[  \mynorm{ \alpha(Q + \delta I + P) + h(\delta) I} \geq \mynorm{\alpha(\psi_{\epsilon}(Q) + \psi_{\epsilon}(P))}.  \]
	By letting $\delta \rightarrow 0$ and using the fact that $h(0+) = 0$, it can be concluded that
	\begin{equation}\label{psi-keeps-projections-orthogonal}
		1 \geq \mynorm{Q + P} \geq \mynorm{\psi_{\epsilon}(Q) + \psi_{\epsilon}(P)}.
	\end{equation}

	Since by Remark \ref{phi-maps-projections-to-projections} $\psi_{\epsilon}(Q)$ and $\psi_{\epsilon}(P)$ are projections, then they are orthogonal to each other. Moreover, if $ m((0,\infty))= 0$, then $\sigma$ is the arithmetic mean and it is clear that (\ref{psi-keeps-projections-orthogonal}) holds.
	\par
	Let $A, B \in \mathcal{A}^{+}$ be such that $AB = 0$ and $\max \lbrace \mynorm{A}, \mynorm{B} \rbrace = t$. Furthermore, denote by $P_A$ and $P_B$ the range projections corresponding to $A$ and $B$. The range projections are elements of $P(\mathcal{A})$ by \cite[Theorem 7]{frank-aw-algebras}. Since $\text{img}(B) \subseteq \text{ker}(A)$, then $ A P_B = 0$; similarly, $P_A P_B = 0$. Since these are orthogonal, then $\psi_{\epsilon}(P_A) \psi_{\epsilon}(P_B) = 0$. Furthermore, $A \leq t P_{A}$, so $\psi_{\epsilon}(A) \leq t\psi_{\epsilon}(P_A)$ by Lemma $\ref{psi-positive-homogenous}$. Thus, $\text{rng}(\psi_{\epsilon}(A)) \subseteq \text{rng}(\psi_{\epsilon}(P_A)) \subseteq \text{ker}(\psi_{\epsilon}(P_B)) \subseteq \text{ker}(\psi_{\epsilon}(B)) $, so $\psi_{\epsilon}(A) \psi_{\epsilon}(B)  = 0$. Thus, each $\psi_{\epsilon}$ extends to a Jordan $*$-isomorphism.
	\par 
	Let $A \in \mathcal{A}^{+}$, and $\epsilon_2 > \epsilon_1$, then 
	\begin{equation*}
		\begin{split}
			\psi_{\epsilon_2}(A)  &= \phi \big( (A + (\epsilon_2 - \epsilon_1) I)  + \epsilon_1 I \big) - \epsilon_2 I \\
			&= \psi_{\epsilon_1} \big( (A + (\epsilon_2 - \epsilon_1) I) \big) + \epsilon_1 I - \epsilon_2 I \\ 
			& = \psi_{\epsilon_1}(A)
		\end{split}
	\end{equation*}
		using the fact that $\psi_{\epsilon_1}$ is linear and unital.
	Therefore, the family of maps $ \lbrace  \psi_{\epsilon} \rbrace_{\epsilon > 0}$ is just one map $\psi$ which extends to a Jordan $*$-isomorphism. Let $A \in \mathcal{A}^{++}$ be such that $A \geq \epsilon I$, then
	\[ \psi(A) = \psi(A - \epsilon I) + \psi(\epsilon I) =  \phi(A), \]
	 so $\phi$ extends to a Jordan $\ast$-isomorphism.
	\end{proof}

	\bibliographystyle{plain}
	\bibliography{kubo_ando_paper_2_for_publishing}

\end{document}